\documentclass{article}
\usepackage[a4paper, left=2cm, right=2cm, top=3cm, bottom=4cm,]{geometry}
\usepackage[utf8]{inputenc}
\usepackage{graphicx}
\usepackage{caption}
\usepackage{amsthm}
\usepackage{amsmath}

\usepackage{todonotes}
\usepackage{hyperref}
\usepackage{rotating}

\newcommand{\DecisionProblem}[3]{
\begin{center}
\begin{tabular}{lp{0.65\textwidth}}
\hline\hline
\multicolumn{2}{c}{\sc #1} \\
\hline
\em{Instance:}& #2\\
\em{Question:}& #3 \\ 
\hline\hline\\
\end{tabular}
\end{center}
}

\newtheorem{theorem}{Theorem}
\newtheorem{definition}{Definition}

\newtheorem{proposition}{Proposition}

\newcommand{\GXC}{G_{X,\mathcal{C}}}
\newcommand{\GX}[1]{G^{X}_{#1}}
\newcommand{\GC}{G^{\mathcal{C}}}

\newcommand{\scomp}{\updownarrow}
\newcommand{\pcomp}{\leftrightarrow}

\title{Geodetic sets for directed acyclic planar geodetic graphs}
\author{Benedikt Georg Hein$\null^1$ and Egon Wanke$\null^2$ \\[12pt]
\small $\null^{1,2}$Heinrich-Heine-Universität Düsseldorf, Germany \\
\small \tt $\null^1$BeHei111@hhu.de, $\null^2$Egon.Wanke@hhu.de}

\begin{document}

\date{}
\maketitle

\begin{abstract}
A set of vertices $S$ of a directed graph $G$ is {\em geodetic} if every vertex of $G$ lies on a shortest path from a vertex of $S$ to a vertex of $S$. A directed graph is {\em geodetic} if there is at most one shortest path from every vertex of $G$ to every vertex of $G$. We prove the NP-completeness of the following decision problem. Given a directed acyclic planar geodetic graph $G$ and an integer $k$, does $G$ have a geodetic set with at most $k$ vertices? This implies that the question of whether $G$ has a strong or a monitoring geodetic set with at most $k$ vertices is also NP-complete for directed acyclic planar geodetic graphs. Furthermore, we prove that the number of vertices in a minimum geodetic set and the number of vertices in a minimum edge geodetic set can be computed in linear time for directed acyclic series-parallel graphs.
\end{abstract}

\centerline{{\bf Keywords :} Geodetic graph, geodetic number, strong geodetic set, monitoring geodetic set}

\maketitle{}

\section{Introduction}

We define various geodetic sets for directed graphs that are determined in a similar way. To avoid repetitions as much as possible, the edge variants are defined using the expressions given in parentheses.

For a graph $G=(V,E)$, let $V(G)=V$ be the set of vertices and $E(G)=E \subseteq V\times V$ be the set of edges of $G$. The length of a path $p$ is the number of its edges. Sometimes we represent paths as a sequence of vertices, sometimes as a sequence of edges, depending on what information about the paths is relevant to our considerations. A vertex $w$ (an edge $e$) is {\em covered by a path} $p$ if $p$ contains the vertex $w$ (the edge $e$). A vertex $w$ (an edge $e$) is {\em covered by a set of vertices} $S$ if there exist two vertices $u,v \in S$ and a shortest path from $u$ to $v$ that covers the vertex $w$ (the edge $e$).

A vertex set $S \subseteq V(G)$ is a {\em geodetic set} (an {\em edge geodetic set}) of $G$ if every vertex $w \in V(G)$ (every edge $e \in E(G)$) is covered by a shortest path from a vertex of $S$ to a vertex of $S$. The vertices $u \in S$ are always covered by the vertices in $S$, because the path consisting of the single vertex $u$ is always a shortest path from $u$ to $u$.

A vertex set $S \subseteq V(G)$ is called a strong geodetic set (strong edge geodetic set) of $G$ if there exists a set of shortest paths $\cal P$, each from some vertex $u \in S$ to some vertex $v \in S$ such that every vertex in $V(G)$ (every edge in $E(G)$) is covered by at least one of these paths. It should be noted here that, for every pair of nodes $u, v \in S$, the set $\cal P$ may contain at most one path from $u$ to $v$.

A vertex set $S \subseteq V(G)$ is a {\em monitoring geodetic set} ({\em monitoring edge geodetic set}) of $G$ if for each vertex $w \in V(G)$ (edge $e \in E(G)$) there is a pair of vertices $u,v \in S$ such that vertex $w$ (edge $e$) is covered by all shortest path from $u$ to $v$.

Let $\text{GS}(G)$, $\text{SGS}(G)$, $\text{MoGS}(G)$, $\text{EGS}(G)$, $\text{SEGS}(G)$ and $\text{MoEGS}(G)$ be the set of all geodetic sets, strong geodetic sets, monitoring geodetic sets, edge geodetic sets, strong edge geodetic sets and monitoring edge geodetic sets, respectively, of a graph $G$.
If every vertex has at least one outgoing or at least one incoming edge then every edge geodetic set is a geodetic set, every strong edge geodetic set is a strong geodetic set and every monitoring edge geodetic set is a monitoring geodetic set. That is, for every graph $G$ satisfying this condition it holds that $\text{EGS}(G) \subseteq \text{GS}(G)$, $\text{SEGS}(G) \subseteq \text{SGS}(G)$ and $\text{MoEGS}(G) \subseteq \text{MoGS}(G)$. 
Furthermore, every strong geodetic (strong edge geodetic) set obviously is a geodetic set (an edge geodetic set) and every monitoring geodetic set (monitoring edge geodetic set) is a strong geodetic set (strong edge geodetic set, respectively). The second statement follows from the following fact: If a vertex $w$ (edge $e$) lies on all shortest paths from a vertex $u$ to a vertex $v$, then we can use any shortest path $p_{u,v}$ from $u$ to $v$ for the definition of a set $\mathcal P$ of paths that forms a strong geodesic set (strong edge geodesic set) of $G$. The union of the vertices $V(p_{u,v})$ (edges $E(p_{u,v})$) of all these shortest paths is the entire set $V(G)$ of vertices ($E(G)$ of edges) of $G$. That is, for every graph $G$ it holds $\text{SGS}(G) \subseteq \text{GS}(G)$, $\text{SEGS}(G) \subseteq \text{EGS}(G)$, $\text{MoGS}(G) \subseteq \text{SGS}(G)$ and $\text{MoEGS}(G) \subseteq \text{SEGS}(G)$.

A graph in which there is at most one shortest path from every vertex $u$ to every vertex $v$ is called a {\em geodetic graph}. In a geodetic graph every geodetic set is a monitoring geodetic set and every edge geodetic set is a monitoring edge geodetic set.

That is, for every graph $G$ it holds $\text{MoGS}(G) = \text{GS}(G)$ and $\text{MoEGS}(G) = \text{EGS}(G)$.
The inclusions of these sets are summarized in Table \ref{Table01} for general directed graphs and in Table \ref{Table02} for geodetic directed graphs.

\begin{table}[hbt]
\center
\begin{tabular}{ccccc}
$\text{MoGS}(G)$ & $\subseteq$ & $\text{SGS}(G)$ & $\subseteq$ & $\text{GS}(G)$ \\
\begin{turn}{90}{$\subseteq$}\end{turn} & & \begin{turn}{90}{$\subseteq$}\end{turn} & & \begin{turn}{90}{$\subseteq$}\end{turn} \\
$\text{MoEGS}(G)$ & $\subseteq$ & $\text{SEGS}(G)$ & $\subseteq$ & $\text{EGS}(G)$ \\
\end{tabular}
\caption{The inclusions of geodetic sets on general directed graphs.}
\label{Table01}
\end{table}

\begin{table}[hbt]
\center
\begin{tabular}{ccccc}
$\text{MoGS}(G)$ & $=$ & $\text{SGS}(G)$ & $=$ & $\text{GS}(G)$ \\
\begin{turn}{90}{$\subseteq$}\end{turn} & & \begin{turn}{90}{$\subseteq$}\end{turn} & & \begin{turn}{90}{$\subseteq$}\end{turn} \\
$\text{MoEGS}(G)$ & $=$ & $\text{SEGS}(G)$ & $=$ & $\text{EGS}(G)$ \\
\end{tabular}
\caption{The inclusions of geodetic sets on geodetic directed graphs.}
\label{Table02}
\end{table}

The {\em (edge) geodetic number}, {\em strong (edge) geodetic number} and {\em monitoring (edge) geodetic number} of a graph $G$ is the minimum number of vertices in a (an edge) geodetic set, strong (edge) geodetic set and monitoring (edge) geodetic set of $G$, respectively.

Geodetic sets were first considered for undirected graphs in \cite{BH1990} and more formally introduced by Harrary et al.~in \cite{HLC1993} and later in \cite{CHZ2002}, see also \cite{HO2007}. A first correct proof of the NP-completeness for deciding whether a given graph $G$ has a geodetic set less than or equal to a given number $k$ can be found in \cite{Ati2002}. This {\sc Minimum Geodetic Set (MGS)} problem is also NP-complete for chordal graphs and bipartite weakly chordal graphs, see \cite{DPRS2010}. The NP-completeness of {\sc MGS} was also shown for interval graphs \cite{CDFGLR2020}, cobipartite graphs \cite{EE2014}, subcubic graphs \cite{BPPRRS2018} and subcubic partial grids of arbitrary girth \cite{CHB2023,CDFGLR2020}. Chakraborty et al.~\cite{CFGGR2020} proved the NP-completeness of {\sc MGS} for undirected planar graphs and line graphs.

Geodetic sets can be computed in polynomial time on ptolemaic\footnote{A ptolemaic graph is an undirected graph whose shortest path distances obey Ptolemy's inequality. That is, for every four vertices $u$, $v$, $w$, and $x$, the inequality $d(u,v)d(w,x) + d(u,x)d(v,w) \geq d(u,w)d(v,x)$ holds.} graphs\cite{FJ1986}, co-graphs, split graphs \cite{DPRS2010}, distance-hereditary graphs \cite{KN2013}, proper interval graphs \cite{EEHHM2012}, outerplanar graphs \cite{Mau2018}, solid grids, k-trees, with fixed k \cite{CDFGLR2020} and complete multipartite graphs \cite{DIT2021}. The problem is also studied for product graphs \cite{BKH2008,FMS2015}, block-cactus graphs \cite{WWC2006}, line graphs \cite{GAVM2012} and median graphs \cite{BH2008}. It has been proven that {\sc MGS} is LOG-APX-hard \cite{DIT2021} and W[1]-hard when parameterized by some graph parameters \cite{KK2022}.

Strong geodetic sets were first considered by Arokiaraj et al.~in \cite{AKMTX2020}. They were studied on cartesian products of graphs \cite{IK2018} and on complete multipartite graphs \cite{IK2019}. In the context of the diameter of a graph, it was discussed by \cite{Irs2018}. It is shown in \cite{Mez2022} that the strong geodetic number of an outerplanar graph can be computed in polynomial time.

Edge geodetic sets were first studied by Atici \cite{Ati2003} and were analyzed for some products of graphs \cite{SC2010,ACC2018}. They were also considered in the context of fuzzy graphs in \cite{RS2019}.
Manuel et al.~\cite{MKXAT2017} showed that the {\sc Minimum Strong Edge Geodetic Set (MSEGS)} problem is NP-complete.

Monitoring edge geodetic sets are defined by Foucaud et al.~in \cite{FNS2023}. Extensive research has been conducted on this topic in recent years, as shown in \cite{Has2023,TLWL2023,MJYL2024,XYBZS2024,FMMSST2025}. Therefore, it is natural to also consider a vertex version of this problem, which is referred to here simply as monitoring geodetic set.

The \textsc{MGS} problem on oriented graphs was studied in \cite{CZ2000}. In this context, the upper and lower geodetic numbers of graphs are often considered \cite{Lu2007,DLW2009,HTW2009}. The oriented version of the monitoring edge geodetic set problem was considered in \cite{DFMDS2025}.

The term "geodetic graph" was first introduced by Ore in \cite{Ore1962} for undirected graphs, see also \cite{BH1990}, and is adopted here for directed graphs. The directed version of a strong geodetic set was first defined in \cite{CTW2004}. 

In this paper, we prove the NP-completeness of the decision problem: given a directed acyclic planar geodetic graph $G$ and an integer $k$, does $G$ have a geodetic set with at most $k$ vertices? This implies that the minimum strong and the minimum monitoring geodetic set problems are also NP-complete for directed acyclic planar geodetic graphs. In Section 4, it is shown that both the geodetic number and the edge geodetic number can be computed in linear time for directed series-parallel graphs. Series-parallel graphs have been long studied in contexts such as electric circuits 
\cite{Duf1965}, complexity of graph algorithms \cite{TNS1982}, and also routing games \cite{Mil2006,EFM2009}.

\section{Preliminaries}

We will first illustrate the definitions of the different geodetic sets using the example of figure \ref{Figure01}. The figure shows a directed graph $G$ that has
\begin{enumerate}
\item a minimum geodetic set $S_1= \{u_5,u_7\}$ which is also a minimum edge geodetic set,
\item a minimum strong geodetic set $S_2= \{u_2,u_5,u_7\}$ verifiable with the shortest paths
$$\begin{array}{lll}
p_{u_5,u_2} & = & u_5,u_1,u_2, \\
p_{u_2,u_7} & = & u_2,u_3,u_4,u_7 \text{ and} \\
p_{u_5,u_7} & = & u_5,u_8,u_6,u_9,u_{10},u_7, \\
\end{array}$$
\item a minimum strong edge geodetic set $S_3= \{u_3,u_5,u_6,u_7\}$ verifiable with the shortest paths
$$\begin{array}{lll}
p_{u_5,u_3} & = & (u_5,u_1),(u_1,u_2),(u_2,u_3), \\
p_{u_5,u_6} & = & (u_5,u_1),(u_1,u_6), \\
p_{u_6,u_7} & = & (u_6,u_3),(u_3,u_4),(u_4,u_7), \\
p_{u_5,u_7} & = & (u_5,u_8),(u_8,u_6),(u_6,u_9),(u_9,,u_{10}),(u_{10},u_7) \text{ and} \\
\end{array}$$
\item a minimum monitoring geodetic set $S_3= \{u_2,u_5,u_7,u_8,u_9\}$ where
\begin{enumerate}
\item $u_1$ is on all shortest paths from $u_5$ to $u_2$,
\item $u_3$ and $u_4$ are on all shortest path from $u_2$ to $u_7$,
\item $u_6$ is on all shortest paths from $u_5$ to $u_9$ and
\item $u_{10}$ is on all shortest path from $u_9$ to $u_7$,
\end{enumerate}
and
\item a minimum monitoring edge geodetic set $S_3= \{u_1,u_2,u_3,u_5,u_7,u_8,u_9\}$ where
\begin{enumerate}
\item $(u_5,u_1)$ is on all shortest paths from $u_5$ to $u_1$,
\item $(u_1,u_2)$ is on all shortest paths from $u_1$ to $u_2$,
\item $(u_2,u_3)$ is on all shortest paths from $u_2$ to $u_3$,
\item $(u_3,u_4)$ and $(u_4,u_7)$ are on all shortest paths from $u_3$ to $u_7$,
\item $(u_5,u_8)$ is on all shortest paths from $u_5$ to $u_8$,
\item $(u_1,u_6)$ and $(u_6,u_9)$ are on all shortest paths from $u_1$ to $u_9$,
\item $(u_8,u_6)$ and $(u_6,u_3)$ are on all shortest paths from $u_8$ to $u_3$ and
\item $(u_9,u_{10})$ and $(u_{10},u_7)$ are on all shortest paths from $u_9$ to $u_7$.
\end{enumerate}
\end{enumerate}
Graph $G$ is not geodetic because, for example, there are two shortest paths from $u_1$ to $u_3$.

\begin{figure}[hbt]
\centerline{\includegraphics[width=170pt]{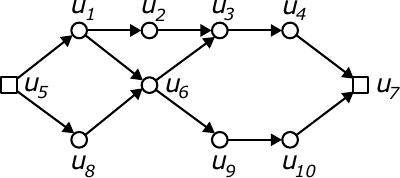}}
\caption{A directed graph $G$. Vertices without incoming edges or without outgoing edges are drawn as squares.}
\label{Figure01}
\end{figure}

The vertices without incoming edges are called {\em terminal start vertices} and the vertices without outgoing edges are called {\em terminal end vertices}. For a directed graph $G$, let $R(G)$ be the set of all terminal start and terminal end vertices. Terminal vertices are only on shortest paths that start or end at them. Vertex set $R(G)$ is a subset of every geodetic set and every edge geodetic set. In the figures we draw terminal vertices as squares.

The decision problems {\sc Minimum Geodetic Set (MGS)}, {\sc Minimum Strong Geodetic Set (MSGS)}, {\sc Minimum Monitoring Geodetic Set (MMoGS)}, {\sc Minimum Edge Geodetic Set (MEGS)}, {\sc Minimum Strong Edge Geodetic Set (MSEGS)} and {\sc Minimum Monitoring Edge Geodetic Set (MMoEGS)} are defined as follows.

\DecisionProblem{Minimum (Strong / Monitoring) (Edge) Geodetic Set}{A graph $G$ and a positive integer $k \leq |V(G)|$.}{Has $G$ a (strong / monitoring) (edge) geodetic set with at most $k$ vertices?}

\medskip
Using a simple polynomial-time reduction from the NP-complete decision problem {\sc Exact Cover by 3 Sets (X3C)}, see \cite{GJ1979}, we show that {\sc MGS} is NP-hard for directed acyclic geodetic graphs. Since $\text{GS}(G)=\text{SGS}(G)=\text{MoGS}(G)$ for geodetic graphs, it follows that {\sc MSGS} and {\sc MMoGS} are also NP-complete for directed acyclic geodetic graphs, where NP-membership follows from the fact that solutions can be verified in polynomial time. These statements are already known, however, we repeat the simple, unified NP-completeness proof because it treats all variants of geodetic sets considered in this paper uniformly. {\sc X3C} is defined as follows.

\DecisionProblem{Exact Cover by 3 Sets (X3C)}{A set $X$ with $3n$ Elements and a collection $\mathcal{C}$ of three-element subsets of $X$.}{Is there a subset $\mathcal{C}'$ of $\mathcal{C}$ with at most $n$ three-element sets such that every element of $X$ occurs in exactly one member of $\mathcal{C}'$?}

\begin{theorem}
{\sc MGS} is NP-complete for directed acyclic geodetic graphs.
\label{Theorem01}
\end{theorem}

\begin{proof}
{\sc MGS} is obviously in NP. We define a directed acyclic geodetic graph $\GXC$ for an {\sc X3C} instance  $X=\{x_1,\ldots,x_{3n}\}$ and $\mathcal{C}=\{C_1,\ldots,C_m\}$ such that $G_{X,\mathcal{C}}$ has a geodetic set with at most $7n+m$ vertices if and only if $X,\mathcal{C}$ has a solution for {\sc X3C}. The vertices and edges of $G_{X,\mathcal{C}}$ are defined as follows (see figure \ref{Figure02} for an example).
$$
V(\GXC) = \{u_1,\ldots,u_m,\,\widetilde{C}_1,\ldots,\widetilde{C}_{m},\,\widetilde{x}_1,\ldots,\widetilde{x}_{3n}\,v_1,\ldots,v_{3n},\,y_1,\ldots,y_{3n},\,w_1,\ldots,w_{3n}\} \\
$$ and
$$
\begin{array}{lll}
E(\GXC) = & & \{(u_j,\widetilde{C}_j)~|~1 \leq j \leq m\} \\
& \cup & \{(u_j,w_i)~|~1 \leq j \leq m, \, 1 \leq i \leq 3n\} \\
& \cup & \{(\widetilde{C}_j,\widetilde{x}_i)~|~1 \leq j \leq m, \, x_i \in C_j\} \\
& \cup & \{(\widetilde{x}_i,v_i)~|~1 \leq i \leq 3n\} \\
& \cup & \{(\widetilde{x}_i,y_i)~|~1 \leq i \leq 3n\} \\
& \cup & \{(y_i,w_i)~|~1 \leq i \leq 3n\} \\
\end{array}
$$
Graph $\GXC$ has a vertex $\widetilde{x_i}$ for every $x_i \in X$ and a vertex $\widetilde{C_j}$ for every $C_j \in \mathcal{C}$. The vertices $v_1,\ldots,v_{3n}$ are necessary to ensure that all vertices $\widetilde{x}_i$ and $\widetilde{C}_{j}$ are covered. While these vertices could alternatively connected directly to the $\widetilde{C_j}$ vertices, but doing so would make the proof harder to follow. It has $6n+m$ terminal vertices $u_1,\ldots,u_j,\,v_1,\ldots,v_{3n},\,w_1,\ldots,w_{3n}$ and can be constructed in polynomial time from instance $X,\mathcal{C}$. It is easy to check that $\GXC$ is geodetic. 

\medskip
\begin{figure}[hbt]
\centerline{\includegraphics[width=417pt]{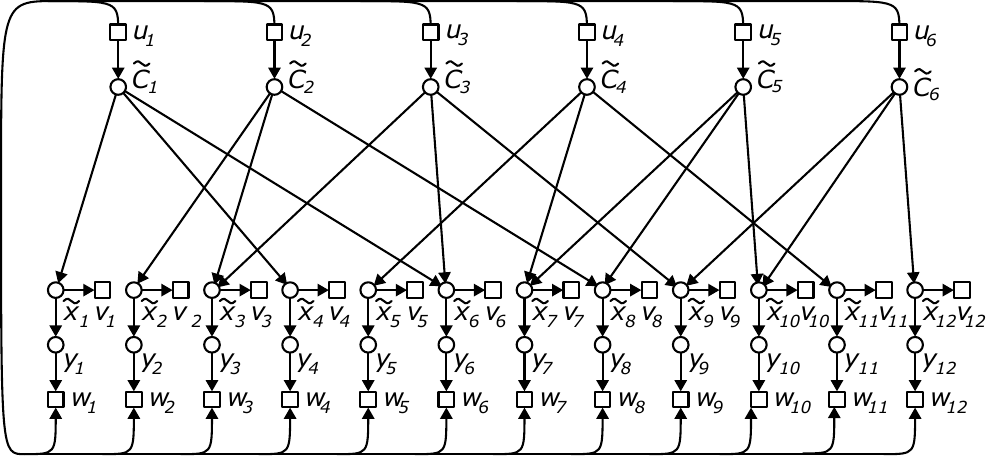}}
\caption{The directed geodetic graph $\GXC$ for $X=\{x_1,\ldots,x_{12}\}$ and $\mathcal{C}=\{C_1,\ldots,C_6\}$ with $C_1=\{x_1,x_4,x_6\}$, $C_2=\{x_2,x_3,x_8\}$, $C_3=\{x_3,x_6,x_9\}$, $C_4=\{x_5,x_7,x_{11}\}$, $C_5=\{x_7,x_8,x_{10}\}$ and $C_6=\{x_9,x_{10},x_{12}\}$.}
\label{Figure02}
\end{figure}
The only vertices of $\GXC$ that are not already covered by shortest paths from terminal start vertices to terminal end vertices are the vertices $y_1,\ldots,y_{3n}$. It remains to show that $\GXC$ has a geodetic set with at most $7n+m$ vertices if and only if $X,\mathcal{C}$ has a solution for {\sc X3C}.

$\Rightarrow$ Suppose there is a subset $\mathcal{C'} \subseteq \mathcal{C}$ with $n$ three-element sets such that every element of $X$ occurs in exactly one member of $\mathcal{C'}$, then vertex set  $S=R(\GXC) \cup \{ \widetilde{C}_{j} ~|~ C_j \in \mathcal{C'}\}$ is a geodetic set with $(6n+m)+n = 7n+m$ vertices.

$\Leftarrow$ Conversely, let $S$ be a geodetic set of $\GXC$ with at most $7n+m$ vertices. Since $\GXC$ has exactly $6n+m$ terminal vertices, the set $S$ has at most $n$ non-terminal vertices. In order to cover the $3n$ vertices $y_1,\ldots,y_{3n}$ with shortest paths, at least $n$ of the $m$ non-terminal vertices $\widetilde{C}_j$, $1 \leq j \leq m$, are necessary. Thus, $S$ has exactly $n$ vertices  $\widetilde{C}_j$ and the set $\{C_j ~|~ \widetilde{C}_j \in S\}$ is a solution for {\sc X3C} for instance $X,\mathcal{C}$.
\end{proof}

\begin{theorem}
{\sc MEGS} is NP-complete for directed acyclic geodetic graphs.
\label{Theorem02}
\end{theorem}

\begin{proof}
The proof is similar to the proof of theorem \ref{Theorem01}. {\sc MEGS} is obviously in NP. We again define a directed acyclic geodetic graph $\GXC$ for an {\sc X3C} instance  $X=\{x_1,\ldots,x_{3n}\}$ and $\mathcal{C}=\{C_1,\ldots,C_m\}$ such that $G_{X,\mathcal{C}}$ has an edge geodetic set with at most $7n+m$ vertices if and only if $X,\mathcal{C}$ has a solution for {\sc X3C}. The vertices and edges of $G_{X,\mathcal{C}}$ are defined as in theorem \ref{Theorem01}. The only difference is the omission of the vertices $y_1,\ldots,v_{3n}$. These vertices are bridged with edges from the vertices $\widetilde{x}_i$ to the vertices $w_i$, see figure \ref{Figure03} for an example.
$$
V(\GXC) = \{u_1,\ldots,u_m,\,\widetilde{C}_1,\ldots,\widetilde{C}_{m},\,\widetilde{x}_1,\ldots,\widetilde{x}_{3n}\,v_1,\ldots,v_{3n},\,w_1,\ldots,w_{3n}\} \\
$$ and
$$
\begin{array}{lll}
E(\GXC) = & & \{(u_j,\widetilde{C}_j)~|~1 \leq j \leq m\} \\
& \cup & \{(u_j,w_i)~|~1 \leq j \leq m, \, 1 \leq i \leq 3n\} \\
& \cup & \{(\widetilde{C}_j,\widetilde{x}_i)~|~1 \leq j \leq m, \, x_i \in C_j\} \\
& \cup & \{(\widetilde{x}_i,v_i)~|~1 \leq i \leq 3n\} \\
& \cup & \{(\widetilde{x}_i,w_i)~|~1 \leq i \leq 3n\} \\
\end{array}
$$
It is easy to see that this graph $\GXC$ is also geodetic. 

\medskip
\begin{figure}[hbt]
\centerline{\includegraphics[width=417pt]{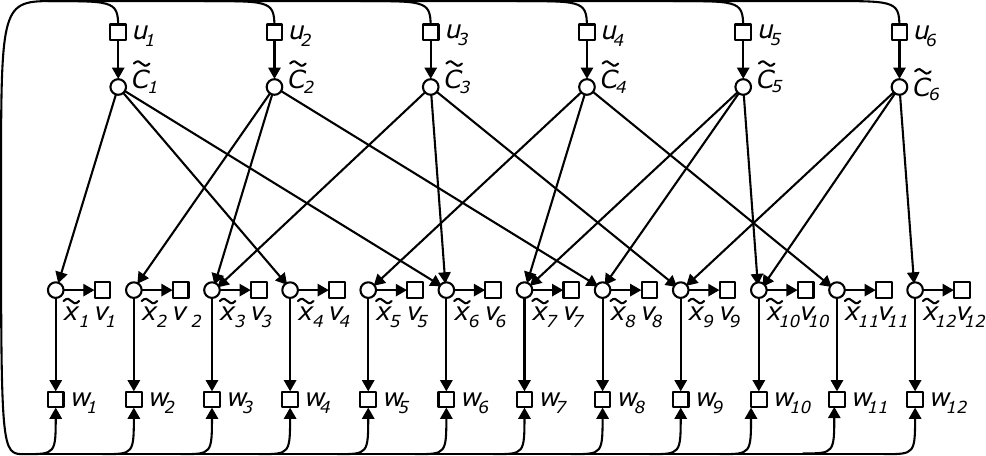}}
\caption{The directed geodetic graph $\GXC$ for $X=\{x_1,\ldots,x_{12}\}$ and $\mathcal{C}=\{C_1,\ldots,C_6\}$ with $C_1=\{x_1,x_4,x_6\}$, $C_2=\{x_2,x_3,x_8\}$, $C_3=\{x_3,x_6,x_9\}$, $C_4=\{x_5,x_7,x_{11}\}$, $C_5=\{x_7,x_8,x_{10}\}$ and $C_6=\{x_9,x_{10},x_{12}\}$.}
\label{Figure03}
\end{figure}

The only edges of $\GXC$ that are not already covered by shortest paths from terminal start vertices to terminal end vertices are the edges $(\widetilde{x}_1,w_1),\ldots,(\widetilde{x}_{3n},w_{3n})$. The argument that $\GXC$ has an edge geodetic set with at most $7n+m$ vertices if and only if $X,\mathcal{C}$ has a solution for {\sc X3C} is the same as in the proof of theorem \ref{Theorem01}.

$\Rightarrow$ Suppose there is a subset $\mathcal{C'} \subseteq \mathcal{C}$ with $n$ three-element sets such that every element of $X$ occurs in exactly one member of $\mathcal{C'}$, then vertex set  $S=R(\GXC) \cup \{ \widetilde{C}_{j,j} ~|~ C_j \in \mathcal{C'}\}$ is an edge geodetic set set with $7n+m$ vertices.

$\Leftarrow$ Conversely, let $S$ be an edge geodetic set of $\GXC$ with at most $7n+m$ vertices. Since $\GXC$ has exactly $6n+m$ terminal vertices, the set $S$ has at most $n$ non-terminal vertices. In order to cover the $3n$ edges $(\widetilde{x}_1,w_1),\ldots,(\widetilde{x}_{3n},w_{3n})$ with shortest paths, at least $n$ of the $m$ non-terminal vertices $\widetilde{C}_j$, $1 \leq j \leq m$, are necessary. Thus, $S$ has exactly $n$ vertices  $\widetilde{C}_j$ and the set $\{C_j ~|~ \widetilde{C}_j \in S\}$ is a solution for {\sc X3C} for instance $X,\mathcal{C}$.
\end{proof}

The next proposition follows from the theorems \ref{Theorem01} and \ref{Theorem02} and the property that in geodetic graphs every (edge) geodetic set is a strong (edge) geodetic set and a monitoring (edge) geodetic set.

\medskip
\begin{proposition}
{\sc MSGS}, {\sc MMoGS}, {\sc MSEGS} and {\sc MMoEGS} are NP-complete for directed acyclic geodetic graphs.
\label{Proposition01}
\end{proposition}

\section{Planar geodetic graphs}

Next we show the NP-completeness of {\sc MGS} for directed acyclic planar geodetic graphs by a polynomial time reduction from {\sc Planar 3-Sat}. The NP-completeness of {\sc Planar 3-Sat} was first shown by Lichtenstein \cite{Lic1982}.

Let $X=\{x_1,\ldots,x_n\}$ be a set of {\em variables} and $L=\{x_1,\overline{x_1},\ldots,x_n,\overline{x_n}\}$ be the {\em literals} of the variables of $X$. A truth assignment of the variables is a mapping $f:X \to \{\text{true},\text{false}\}$. Literal $x_i$ is {\em true under} $f$ if $f(x_i) = \text{true}$, it is {\em false under} $f$ if $f(x_i) = \text{false}$. Literal $\overline{x_i}$ is {\em true under} $f$ if $f(x_i) = \text{false}$, it is {\em false under} $f$ if $f(x_i) = \text{true}$. A {\em clause} $C \subseteq L$ is a set of literals. A truth assignment $f$ satisfies a clause $C_j$ if $C_j$ has a literal that is true under $f$.

For a set of variables $X$ and a set of clauses $\mathcal{C}=\{C_1,\ldots,C_m\}$, the variable-clause graph for $X,\mathcal{C}$ is the bipartite undirected graph with vertex set $\{\widetilde{x}_1,\ldots,\widetilde{x}_n\} \cup \{\widetilde{C}_1,\ldots,\widetilde{C}_n\}$ and an edge between variable vertex $\widetilde{x}_i$ and clause vertex $\widetilde{C}_j$ if and only if clause $C_j$ contains a literal for variable $x_i$. See figure \ref{Figure04} for an example.

\begin{figure}[hbt]
\centerline{\includegraphics[width=115pt]{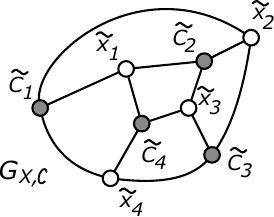}}
\caption{The bipartite planar variable-clause graph for $X = \{x_1,x_2,x_3,x_4\}$ and $\mathcal{C}=\{C_1,C_2,C_3,C_4\}$ with $C_1=\{x_1,x_2,\overline{x_4}\}$, $C_2=\{x_1,\overline{x_2},x_3\}$, $C_3=\{\overline{x_2},x_3,x_4\}$ and $C_4=\{\overline{x_1},\overline{x_3},x_4\}$. The white-colored vertices are the vertices of the variables. The gray-colored vertices are the vertices of the clauses.}
\label{Figure04}
\end{figure}

The decision problem {\sc Planar 3-Sat} is defined as follows.

\DecisionProblem{Planar 3-Sat}{A set of variables $X=\{x_1,\ldots,x_n\}$ and a set of clauses $\mathcal{C}=\{C_1,\ldots,C_m\}$ over $X$ such that each clause $C_j \in \mathcal{C}$ has exactly three literals and the variable-clause graph $\GXC$ for $X,\mathcal{C}$ is planar.}
{Is there a truth assignment $f:X \to \{\text{true},\text{false}\}$ that satisfies every clause $C_j \in \mathcal{C}$.}

\begin{theorem}
{\sc MGS} is NP-complete for directed acyclic geodetic planar graphs.
\label{Theorem03}
\end{theorem}

\begin{proof}
{\sc MGS} in NP because for a given set of vertices $S$ it is easy to check that $S$ is a geodetic set.

We now show how to create, in polynomial time, a directed acyclic planar geodetic graph $\GXC$ for a given {\sc Planar 3-Sat} instance $X,\mathcal{C}$ with $m$ clauses. The Graph $\GXC$ has a geodetic set containing at most $29 m$ vertices if and only if there is a truth assignment $f:X \to \{\text{true},\text{false}\}$ that satisfies every clause $C_j \in \mathcal{C}$.

The procedure is as follows. We design a gadget graph $\GX{r}$ for the variables and a gadget graph $\GC$ for the clauses.  Let $r_i$ be the total number of literals for variable $x_i$ in all clauses $C_1,\ldots,C_m$. Then in graph $\GXC$ there is a copy $G_{x_i}$ of gadget $\GX{r_i}$ for each variable $x_i \in X$ and a copy $G_{C_j}$ of gadget $\GC$ for each clause $C_j \in \mathcal{C}$. Copy $G_{x_i}$ is connected to copy $G_{C_j}$ via directed edges from vertices of $G_{x_i}$ to vertices of $G_{C_j}$ if and only if clause $C_j$ contains a literal for variable $x_i$. The final resulting graph $\GXC$ will be acyclic planar and geodetic and has a geodetic set with at most $29 m$ vertices if and only if there is a truth assignment $f:X \to \{\text{true},\text{false}\}$ for $X$ that satisfies every clause $C_j \in \mathcal{C}$.

\medskip
\begin{itemize}
\item {\bf The definition of variable gadget $\GX{r}$:}

\medskip
Let $J$ be the graph with vertex set $V(J)=\{T,F,O,I_1,\ldots,I_6,u_1,\ldots,u_7\}$ and edge set
$$
E(J)=\left\{
\begin{array}{l}
(O,T),(O,F),(O,u_4),(O,u_7),
(T,I_1),\\
(u_1,T),(u_1,I_2),\\
(u_2,F),(u_2,I_3),\\
(u_3,I_1),\\
(u_4,u_1),(u_4,u_3),(u_4,u_5),(u_4,I_5),\\
(u_5,I_5),(u_5,I_6),\\
(u_6,u_5),(u_6,I_4),\\(u_7,u_2),(u_7,u_6),(u_7,I_6)
\end{array}\right\},
$$
see also figure \ref{Figure05}.

\medskip
\begin{figure}[hbt]
\centerline{\includegraphics[width=205pt]{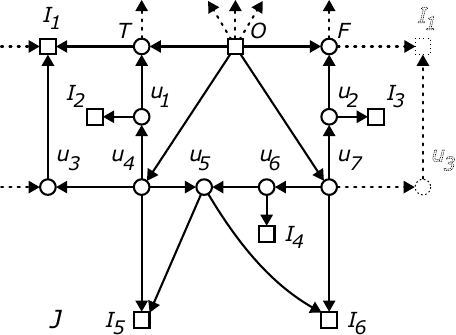}}
\caption{The graph $J$.}
\label{Figure05}
\end{figure}

\medskip
Variable gadget $\GX{r}$ is composed from $r$ copies $J_1,\ldots,J_r$ of graph $J$ in a circular manner as shown in figure \ref{Figure06}. Let
$$T^{(k)},F^{(k)},O^{(k)},I_1^{(k)},\ldots,I_6^{(k)},u_1^{(k)},\ldots,u_7^{(k)}$$
be the vertices 
$$T,F,O,I_1,\ldots,I_6,u_1,\ldots,u_7$$
of the $k$-th copy $J_k$, $1 \leq k \leq r$, in $\GX{r}$ and
let
$$
\begin{array}{lll}
T^{(\star)} & = & \{T^{(k)}~|~1 \leq k \leq r\}, \\
F^{(\star)} & = & \{F^{(k)}~|~1 \leq k \leq r\}, \\
O^{(\star)} & = & \{O^{(k)}~|~1 \leq k \leq r\}, \\
u_3^{(\star)} & = & \{u_3^{(k)}~|~1 \leq k \leq r\}, \\
u_4^{(\star)} & = & \{u_4^{(k)}~|~1 \leq k \leq r\}, \\
u_5^{(\star)} & = & \{u_5^{(k)}~|~1 \leq k \leq r\}, \\
u_7^{(\star)} & = & \{u_7^{(k)}~|~1 \leq k \leq r\}
\end{array}
$$
The $r$ copies of $J$ are connected by additional edges. These are the edges from $F^{(k)}$ to $I^{(k')}_1$ and the edges from $u^{(k)}_7$ to $u^{(k')}_3$ where $k' = (k \bmod r)+1$. 

\begin{figure}[hbt]
\centerline{\includegraphics[width=412pt]{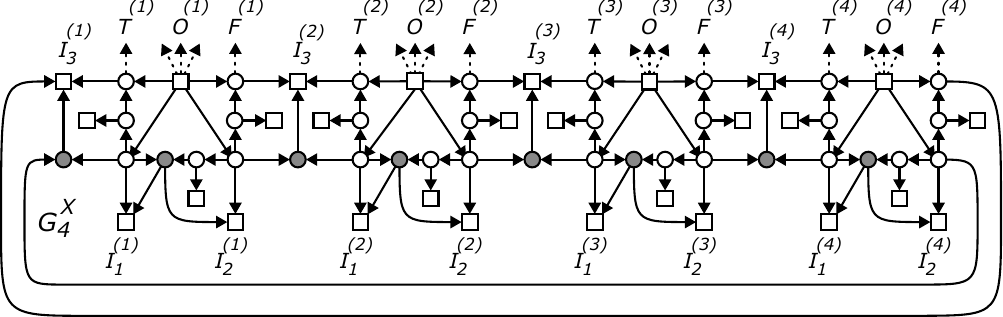}}
\caption{The figure shows variable gadget $\GX{4}$. It is used in the case where there are exactly four literals for a variable $x_i$ in all clauses together. The gray-colored vertices are the vertices not covered by shortest paths from a terminal start to a terminal end vertex.}
\label{Figure06}
\end{figure}

\medskip
The dashed lines in figure \ref{Figure06} show the vertices and edges of the neighboring copies of $J$ in gadget $\GX{r}$ and the possible outgoing edges to a copy of a clause gadget $\GC$ in $\GXC$. The dashed lines in figure \ref{Figure06} shows the possible outgoing edges from vertices of a copy of $\GX{r}$ to vertices of a copy of a clause gadget. Only the vertices of $O^{(\star)} \cup T^{(\star)} \cup F^{(\star)}$ will get additional outgoing edges in $\GXC$.

\medskip
Graph $J$ has one terminal start vertex $O$ and 6 terminal end vertices $I_1,\ldots,I_6$. Thus variable gadget $\GX{r}$ has $r$ terminal start vertices and $6 r$ terminal end vertices. The terminal start vertices of the copies of $\GX{r}$ will also be terminal start vertices in $\GXC$, because there will be no additional edges in $\GXC$ to the vertices of these copies. The terminal end vertices of the copies of $\GX{r}$ will also be terminal end vertices in $\GXC$, because there will be no additional edges in $\GXC$ from these vertices of these copies.

\medskip
It is easy to check that the shortest paths from the terminal start vertices to the terminal end vertices of $\GX{r}$ cover all vertices of $\GX{r}$ except the $2 r$ vertices of $u_3^{(\star)} \cup u_5^{(\star)}$. These vertices are colored gray in figure \ref{Figure06}.

\medskip
It is also easy to check that variable gadget $\GX{r}$ is acyclic, planar and geodetic, and has exactly two minimum geodetic sets $R(\GX{r}) \cup u_4^{(\star)}$ and $R(\GX{r}) \cup u_7^{(\star)}$ with $8 r$ vertices. The first set is called a minimum geodetic set of type {\em true} and the second set is called a minimum geodetic set of type {\em false}.

\medskip
Another interesting property of the variable gadget $\GX{r}$ is the following. For each $k$, $1 \leq k \leq r$, there is

\begin{enumerate}
\item a path from vertex $u_4^{(k)}$ to vertex of $T^{(k)}$ and no path from $u_4^{(k)}$ to a vertex of $(T^{(\star)} \cup F^{(\star)}) \setminus \{T^{(k)}\}$ and
\item a path from vertex $u_7^{(k)}$ to vertex $F^{(k)}$ and no path from $u_7^{(k)}$ to a vertex of $(T^{(\star)} \cup F^{(\star)}) \setminus \{F^{(k)}\}$.
\end{enumerate}

\medskip
The vertex triples $(T^{(k)},O^{(k)},F^{(k)})$, $1 \leq k \leq r$, are called the {\em variable ports} of $\GX{r}$. For each $k$, $1 \leq k \leq r$, either the two vertices $T^{(k)},O^{(k)}$ or the two vertices $O^{(k)},F^{(k)}$ will be connected in $\GXC$ via directed edges to vertices from a clause gadget $\GC$.

\bigskip
\item {\bf The definition of clause gadget $\GC$:}

\medskip
The clause gadget $\GC$ has vertex set $$V(\GC)=\{\hat{I}_1,\,\hat{I}_2,\,\hat{I}_3,\,\hat{I}_4,\,\hat{I}_5,\,v_{1,1},\,v_{2,1},\,v_{2,2},\,v_{3,1},\,v_{3,2},\,v_{3,3},\,w\}$$
and edge set

$$E(\GC) = \left\{
\begin{array}{l}
(v_{1,1},\hat{I}_3),(v_{1,1},\hat{I}_1),(v_{2,1},\hat{I}_1), (v_{2,1},\hat{I}_2),(v_{3,1},\hat{I}_2),(v_{3,1},\hat{I}_3),\\
(v_{1,1},w),\\
(v_{2,1},v_{2,2}),(v_{2,2},w),(v_{2,2},\hat{I}_5),\\
(v_{3,1},v_{3,2}),(v_{3,2},v_{3,3}),(v_{3,3},w),(v_{3,3},\hat{I}_4), \\
(w,\hat{I}_1),(w,\hat{I}_2),(w,\hat{I}_3)
\end{array}
\right\},$$
see figure \ref{Figure07}. The vertices $\hat{I}_1, \ldots, \hat{I}_5$ are terminal end vertices in the clause gadget and also in the final graph $\GXC$, because there are no outgoing edges from vertices of a copy of a clause gadget in $\GXC$. The dashed lines in figure \ref{Figure07} indicate all possible incoming edges from vertices of a copy of a variable gadget in $\GXC$.

\begin{figure}[hbt]
\centerline{\includegraphics[width=213pt]{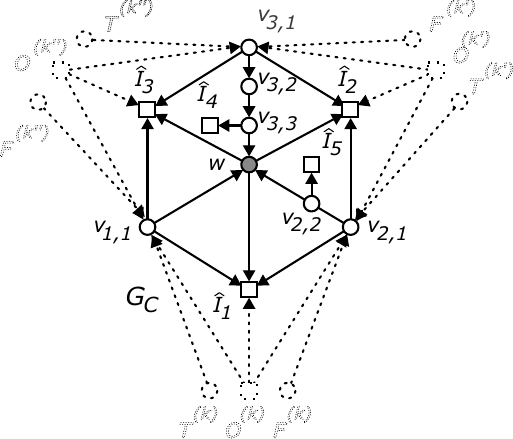}}
\caption{The clause gadget $\GC$.}
\label{Figure07}
\end{figure}

\medskip
The planar embedding given in figure \ref{Figure07} defines in a circle around to the left along the outer face three vertex triples $(v_{1,1},\hat{I}_1,v_{2,1})$, $(v_{2,1},\hat{I}_2,v_{3,1})$ and $(v_{3,1},\hat{I}_3,v_{1,1})$. These vertex triples are called {\em clause ports}.

\bigskip
\item {\bf The definition of graph $\GXC$:}

\medskip
The final graph $\GXC$ has for each variable $x_i \in X$ a copy $G_{x_i}$ of variable gadget $\GX{r_i}$ and for each clause $C_j \in \mathcal{C}$ a copy $G_{C_j}$ of clause gadget $\GC$.

\medskip
Figure \ref{Figure08} shows the connections between the variable graphs $G_{x_i}$ and the clause graphs $G_{C_j}$.

\medskip
If clause $C_j$ contains a literal for variable $x_i$ then 4 edges are inserted from the vertices of a clause port $(T^{(k)},O^{(k)},F^{(k)})$ of $G_{x_i}$ to the vertices of a variable port $(v_{l,1},\hat{I}_l,v_{l'.1})$ of $G_{C_j}$. When choosing the variable port and clause port, the planar embedding of the variable clause graph of the {\sc Planar-3-Sat} instance $X,\mathcal{C}$ must be taken into account.

\medskip
If literal $x_i \in C_j$ then the 4 edges $$(T^{(k)},v_{l,1}),~(O^{(k)},v_{l,1}),~(O^{(k)},\hat{I}_{l,1}),~(O^{(k)},v_{l',1})$$ are inserted, if literal $\overline{x_i} \in C_j$ then the 4 edges $$(O^{(k)},v_{l,1}),~(O^{(k)},\hat{I}_{l,1}),~(O^{(k)},v_{l',1}),~(F^{(k)},v_{l',1})$$ are inserted.

\medskip
\begin{figure}[hbt]
\centerline{\includegraphics[width=410pt]{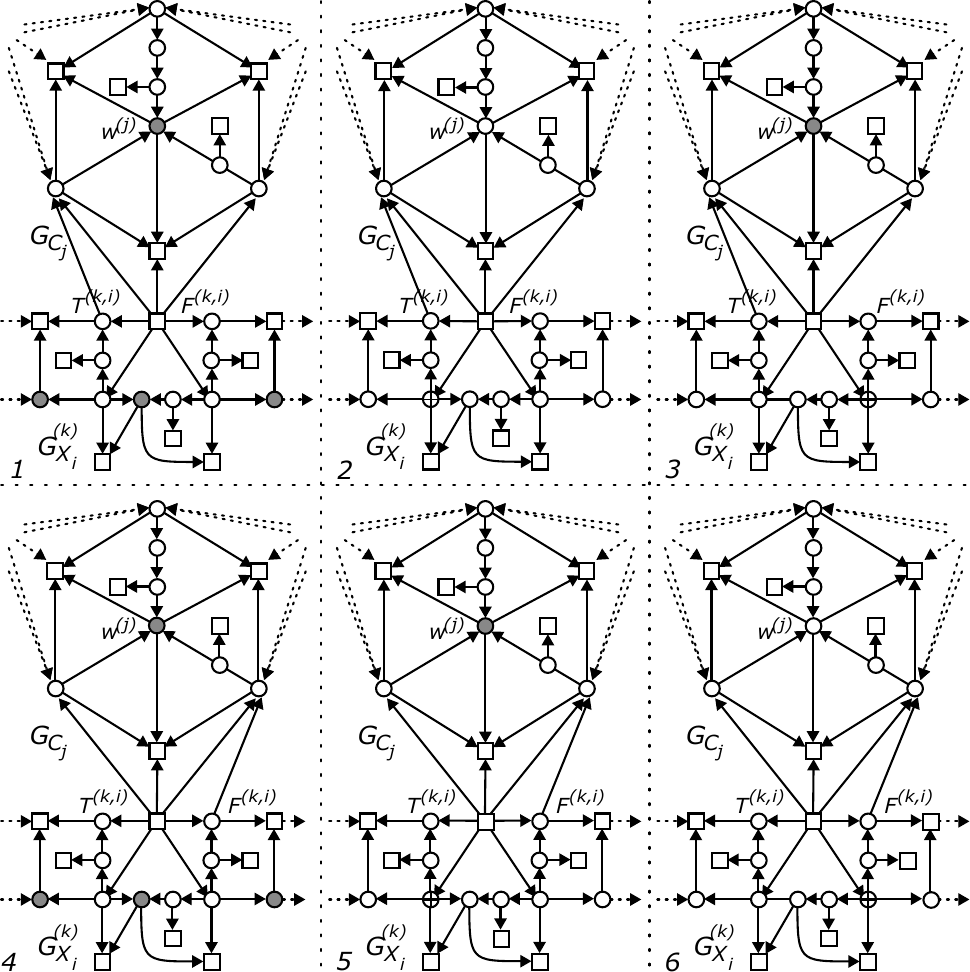}}
\caption{The possible edges from a variable graph $G_{x_i}$ to a clause graph $G_{C_j}$.}
\label{Figure08}
\end{figure}

\begin{itemize}
\item 
The top row of figure \ref{Figure08} (tiles 1, 2 and 3) shows the case in which literal $x_i$ is in $C_j$. In this case there is an edge from vertex $T^{(k)}$ to vertex $v_{l,1}$ of $G_{C_j}$.
\item
The bottom row of figure \ref{Figure08} (tiles 4, 5 and 6) shows the case in which literal $\overline{x_i}$ is in $C_j$.In this case there is an edge from vertex $F^{(k)}$ to vertex $v_{l',1}$.
\item
In the first column of figure \ref{Figure08} (tiles 1 and 4) the vertices not covered by $R(\GXC)$ are colored gray.
\item
In the second column of figure \ref{Figure08} (tiles 2 and 5), the vertices not covered by a minimum geodetic set of type true for $G_{x_i}$ are colored gray.
\item
In the third column of figure \ref{Figure08} (tiles 3 and 6), the vertices not covered by a minimum geodetic set of type false for $G_{x_i}$ are colored gray.
\end{itemize}
In the first case where literal $x_i \in C_j$, vertex $w$ from $G_{C_j}$
is covered by the geodetic set of type true.
This is because there is a shortest path from $u_4^{(k)}$ via $u_1^{(k)}$ to $T^{(k)}$ in the copy of $J$, see figure \ref{Figure05}, and then either via the vertices $v_{1,1}$ and $w$ to vertex $\hat{I}_{2}$, or via the vertices $v_{2,1}$, $v_{2,2}$ and $w$ to vertex $\hat{I}_{3}$ or via the vertices $v_{3,1}$, $v_{3,2}$, $v_{3,3}$ and $w$ to vertex $\hat{I}_{1}$, respectively, see figure \ref{Figure08}. 
Which cases actually occur depends on which clause port is connected to which variable port.
 
\medskip
In the second case where literal $\overline{x_i} \in C_j$, vertex $w$ from $G_{C_j}$ is covered by the geodetic set of type false.
This is because there is a shortest path from $u_7^{(k)}$ via $u_2^{(k)}$ to $F^{(k)}$ in the copy of $J$, see figure \ref{Figure05}, and then either via the vertices $v_{1,1}$ and $w$ to vertex $\hat{I}_{2}$, or via the vertices $v_{2,1}$, $v_{2,2}$ and $w$ to vertex $\hat{I}_{3}$ or via the vertices $v_{3,1}$, $v_{3,2}$, $v_{3,3}$ and $w$ to vertex $\hat{I}_{1}$, respectively, see figure \ref{Figure08}.

\medskip
Graph $\GXC$ is acyclic, planar and geodetic.

\end{itemize}

It remains to show that if there is a truth assignment $f:X \to \{\text{true},\text{false}\}$ that satisfies every clause $C_j \in \mathcal{C}$ then $\GXC$ has a geodetic set with at most $29 m$ vertices and if $\GXC$ has a geodetic set with at most $29 m$ vertices then there is a truth assignment $f:X \to \{\text{true},\text{false}\}$ that satisfies every clause $C_j \in \mathcal{C}$.

$\Rightarrow$ Suppose $f:X \to \{\text{true},\text{false}\}$ is a truth assignment for $X$ that satisfies every clause $C_j \in \mathcal{C}$. Let $S_{x_i}$ be a minimum geodetic set of copy $G_{x_i}$ of type true if $f(x_i) = \text{true}$, or of type false if $f(x_i) = \text{false}$. Then the union of all $S_{x_i}$, $x_i \in X$, and the the set of all terminal vertices of all $G_{C_j}$ is a geodetic set with $29 m$ vertices.

$\Leftarrow$ Suppose $S$ is a geodetic set of $\GXC$ with at most $29 m$ vertices. Let
$$u_{3}^{(k,i)},u_{4}^{(k,i)},u_{5}^{(k,i)},u_{7}^{(k,i)},$$
$1 \leq k \leq r_i$, be the vertices
$$u_{3}^{(k)},u_{4}^{(k)},u_{5}^{(k)},u_{7}^{(k)}$$
of gadget $\GX{r_i}$ in graph $G_{x_i}$ of $\GXC$, and let $w^{(j)}$, $1 \leq j \leq m$, be the vertex $w$ of gadget $\GC$ in graph $G_{C_j}$ of $\GXC$.
Vertex set $S$ must contain the $7 \cdot 3 m = 21 m$ terminal vertices of all $G_{x_i}$, $x_i \in X$, and the $5 \cdot m$ terminal vertices of all $G_{C_j}$, $C_j \in \mathcal{C}$. In addition, $S$ must contain $1 \cdot 3 m$ vertices from the set
$$\{u_{4}^{(k,i)},u_{7}^{(k,i)}~|~1 \leq k \leq r_i, 1 \leq i \leq n\}$$
to cover all vertices $u_{3}^{(k,i)}$ and $u_{5}^{(k,i)}$ for $1 \leq k \leq r_i$ and $1 \leq i \leq n$.
More precisely, $S$ must contain for each $x_i \in X$ all vertices $u_{4}^{(k,i)}$ or all vertices $u_{7}^{(k,i)}$, $1 \leq k \leq r_i$. Since $S$ covers all $m$ vertices $w^{(j)}$ of $G_{C_j}$, $C_j \in \mathcal{C}$, the truth assignment $f:X \to \{\text{true},\text{false}\}$ defined by
$$f(x_i) = \left\{
\begin{array}{ll}
\text{true}  & \text{if } S \cap V(G_{x_i}) \text{ is a minimum geodetic set of } G_{x_i} \text{ of type true}\\
\text{false} & \text{if } S \cap V(G_{x_i}) \text{is a minimum geodetic set of } G_{x_i} \text{ of type false}
\end{array}
\right.$$
satisfies every clause $C_j \in \mathcal{C}$.

The transformation can obviously be done in polynomial time.
\end{proof}

The following proposition is a consequence of theorem \ref{Theorem03} and the property that in geodetic graphs a set of vertices is geodetic if and only if it is strong geodetic if and only if it is monitoring geodetic.

\begin{proposition}
{\sc MSGS} and {\sc MMoGS} are NP-complete for directed acyclic geodetic planar graphs.
\label{Proposition02}
\end{proposition}

Figure \ref{Figure09} shows a complete example of such a Graph $\GXC$ constructed based on a planar 3-SAT instance. The structure of the corresponding planar embedded variable-clause graph remains clearly visible.

\begin{figure}[hbt]
\centerline{\includegraphics[width=410pt]{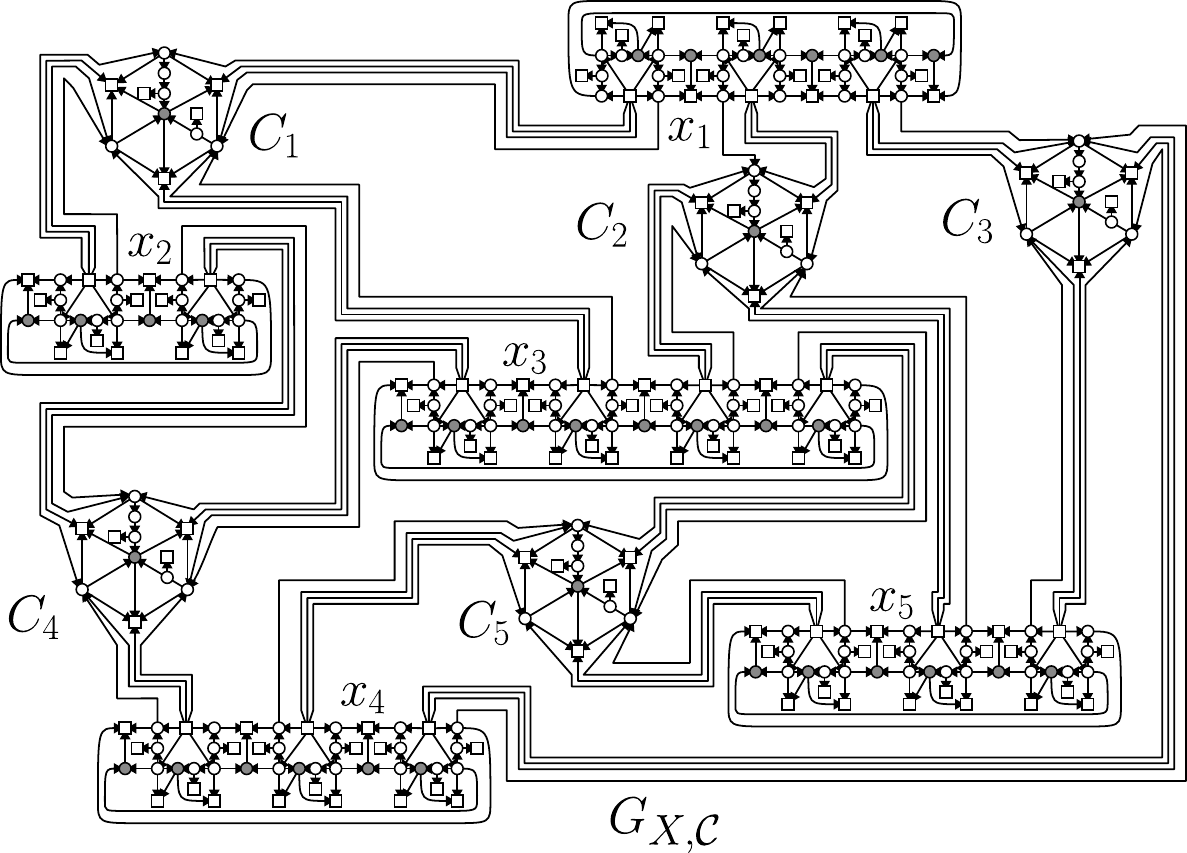}}
\caption{Example of Graph $\GXC$ derived from the planar 3-SAT instance $X = \{x_1,x_2,x_3,x_4,x_5\}$ and $\mathcal{C} =\{C_1,C_2,C_3,C_4,C_5\}$ with $C_1=(x_1,\overline{x_2},\overline{x_3 })$, $C_2=(\overline{x_1},\overline{x_3},\overline{x_5})$, $C_3=(x_1,\overline{x_4},x_5)$, $C_4=(x_2,x_3,x_4)$ and $C_5=(x_3,x_4,\overline{x_5})$}
\label{Figure09}
\end{figure}

\section{Directed Series Parallel Graphs}

Next, we show that minimum geodetic sets and minimum edge geodetic sets can be computed in linear time for {\em directed series-parallel graphs} which we also call {\em series-parallel digraphs}. Series-parallel graphs (in their undirected version) have a tree width of at most 2, but not all of them are outerplanar. They are recursively defined as follows.

\begin{definition}[series-parallel diraphs, \cite{Epp1992}]
A {\em series-parallel digraph} is a system $G=(V,E,s,t)$ where $(V,E)$ is a directed graph and $s,t\in V$ are two distinct vertices. Vertex $s$ is called the {\em source vertex} and vertex $t$ is called the {\em target vertex} of $G$. For an edge set $E \subseteq V \times V$ and two vertices $u', u \in V$, let
$$\text{merge}_{u'\to u}(E) = 
\begin{array}{ll}
     & \{ (v,w) ~|~ v \not= u' \wedge w \not=u' \wedge (v,w) \in E \} \\
\cup & \{ (v,u) ~|~ v \not= u' \wedge (v,u') \in E \} \\
\cup & \{ (u,w) ~|~ w \not= u' \wedge (u',w) \in E \}.
\end{array}
$$
We say vertex $u'$ and vertex $u$ are merged to vertex $u$. \\
The following recursive definition generates the class of series-parallel digraphs.

\begin{enumerate}
    \item $I_{(s,t)}=(\{s,t\},\{(s,t)\},s,t)$ is a series-parallel digraph with source vertex $s$ and target vertex $t$. It is called the {\em initial series-parallel digraph}.

    \item If $G=(V,E,s,t)$ and $G'=(V',E',s',t')$ are two series-parallel digraphs with source vertices $s,s'$ and target vertices $t,t'$, respectively, then
    \begin{enumerate}
        \item $G \scomp G'$ with vertex set $V \cup V' \setminus \{s'\}$ and edge set $$\text{merge}_{s'\to t}(E \cup E')$$ is a series-parallel digraph with source vertex $s$ and target vertex $t'$ and

        \item $G \pcomp G'$ with vertex set $V \cup V' \setminus \{s',t'\}$ and edge set $$\text{merge}_{s'\to s}(\text{merge}_{t'\to t}(E \cup E'))$$ is a series-parallel digraph with source vertex $s$ and target vertex $t$.

    \end{enumerate}
\end{enumerate}

Operation $\scomp$ is called a {\em series composition} and operation $\pcomp$ is called a {\em parallel composition} of two series parallel graphs.

\end{definition}

Series-parallel digraphs $G=(V,E,s,t)$ are planar and cycle-free. For every vertex $u \in V$, there is a path from $s$ to $u$ and a path from $u$ to $t$. If the directed edges are considered undirected, then series-parallel digraphs have tree-width of at most 2.
Readers interested in a deeper understanding of this graph class are referred to \cite{Aye1978} and \cite{VTL1982}. An example is shown in Figure \ref{Figure12}.

A {\em subdivision} of a directed graph $G$ is a graph obtained by replacing directed edges $(u,v)$ in $G$ by a new vertex $w$ with two new directed edges $(u,w)$ and $(w,v)$.
Series-parallel digraphs do not contain the {\em diamond digraph} $\vec{D}$ from figure \ref{Figure10} as a subgraph. They also do not contain a subdivision of $\vec{D}$ as a subgraph, because these graphs cannot be constructed by either a series or a parallel composition of two series-parallel digraphs.

\begin{figure}[hbt]
\centerline{\includegraphics[width=250pt]{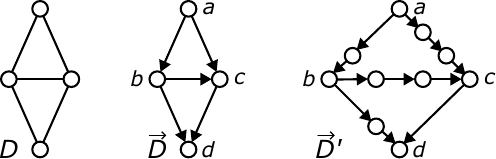}}
\caption{The {\em diamond graph} $D$, a {\em diamond digraph} $\vec{D}$ and a subdivision $\vec{D}'$ of $\vec{D}$.
The graphs $\vec{D}$ and $\vec{D}'$ are not subgraphs of series-parallel digraphs.}
\label{Figure10}
\end{figure}
\begin{theorem}
If a vertex set $S$ is a geodetic set (edge geodetic set) of a series-parallel digraph $G=(V,E,s,t)$ then every vertex $w \in V \setminus S$ (edge $e \in E$) lies on a shortest path from source vertex $s$ to a vertex of $S$ or on a shortest path from a vertex of $S$ to target vertex $t$.
\label{Theorem04}
\end{theorem}

\begin{figure}
\centerline{\includegraphics[width=370pt]{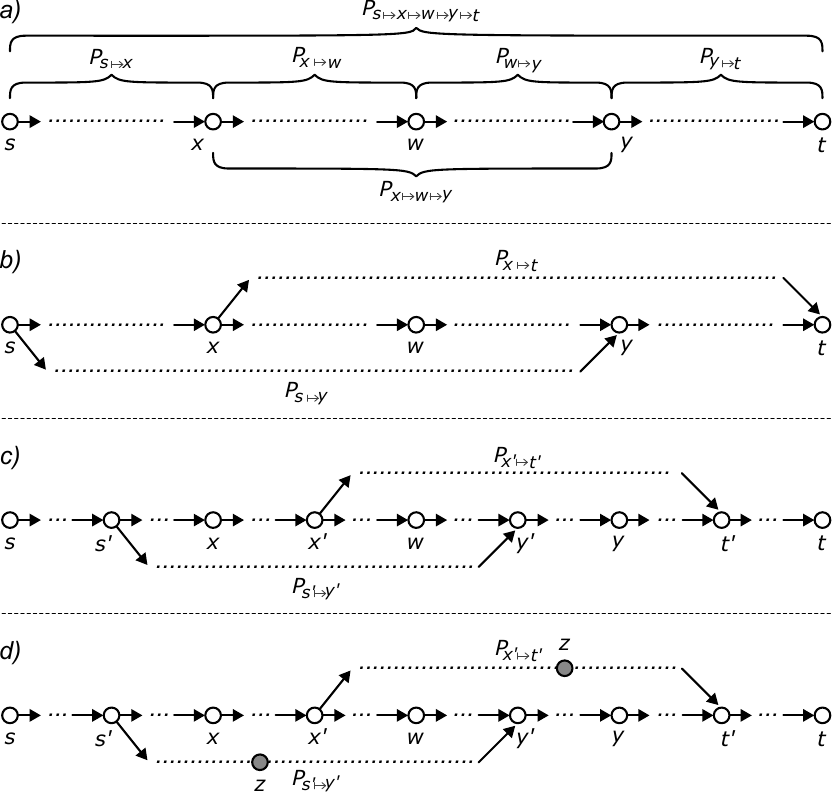}}
\caption{The graphs defined by the paths $P_{s\mapsto x}$, $P_{x\mapsto w}$, $P_{w\mapsto y}$, $P_{y\mapsto t}$, $P_{s\mapsto y}$, $P_{x\mapsto t}$, $P_{s'\mapsto y'}$ and $P_{x'\mapsto t'}$ as used in the proof of theorem \ref{Theorem04}.}
\label{Figure11}
\end{figure}

\begin{proof}
We first consider the case for geodetic sets. After that we consider the case for edge geodetic sets.

We prove the theorem by contradiction. For two vertices $x,y$ contained in a geodetic set, any third vertex $w$ is never exclusively covered by the shortest path between $x$ and $y$. If a vertex $w \in V \setminus S$ lies not on a shortest path from vertex $s$ to a vertex $y \in S$ and not on a shortest path from a vertex $x \in S$ to $t$, then a subdivision of the diamond digraph $\vec{D}$ from figure \ref{Figure10} is a subgraph of $G$, which contradicts the property that series-parallel graphs do not contain a subdivision of $\vec{D}$ as a subgraph.

Let $x\not=s$, $y\not=t$ and $P_{s \mapsto x \mapsto w \mapsto y \mapsto t}$ be a path in $G$ from vertex $s$ via vertex $x$ via vertex $w$ via vertex $y$ to vertex $t$, where the subpath $P_{x \mapsto w \mapsto y}$ of $P_{s \mapsto x \mapsto w \mapsto y \mapsto t}$ from $x$ via $w$ to $y$ is a shortest path. See Graph a) in figure \ref{Figure11}. Every path in $G$ is simple. A path cannot contain a vertex more than once, since $G$ is cycle-free.

Next let $P_{s \mapsto y}$ be a shortest path from $s$ to $y$ and $P_{x \mapsto t}$ be a shortest paths from $x$ to $t$. See Graph b) in figure \ref{Figure11}.
Path $P_{s \mapsto y}$ does not contain any vertex of path $P_{x \mapsto w}$, since otherwise there exists a shortest path from $s$ to $y$ containing vertex $w$, which contradicts our assumptions. For the same reason, path $P_{x \mapsto t}$ does not contain any vertex of path $P_{w \mapsto y}$.

Therefore,
let $s'$ be the last vertex of $P_{s \mapsto y}$ that is also in $P_{s \mapsto x}$ and $y'$ be the first vertex of $P_{s \mapsto y}$ that is also in $P_{w \mapsto y}$. 
Let $x'$ be the last vertex of $P_{x \mapsto t}$ that is also in $P_{x \mapsto w}$ and $y'$ be the first vertex of $P_{x \mapsto t}$ that is also in $P_{y \mapsto t}$. See Graph c) in figure \ref{Figure11}.

If the two subpaths $P_{s' \mapsto y'}$ and $P_{x' \mapsto t'}$ of $P_{s \mapsto y}$ and $P_{x \mapsto t}$, respectively, are vertex disjoint, then $G$ has a subdivision of diamond digraph $\vec{D}$ of figure \ref{Figure10}, where $a=s'$, $b=x'$, $c=y'$ and $d=t'$. See Graph c) in figure \ref{Figure11}. Suppose $P_{s' \mapsto y'}$ and $P_{x' \mapsto t'}$ have a common vertex. Then let $z$ be the first common vertex of these two path. In this case, $G$ also contains a subdivision of the diamond digraph $\vec{D}$, where $a=s'$, $b=x'$, $c=z$, and $d=y'$. See Graph d) in figure \ref{Figure11}.

The proof for edge geodetic sets is analogous. Let $(w,w')$ be an edge on a shortest path from $x$ to $y$. Then we consider the path $P_{s \mapsto x \mapsto w, w' \mapsto y \mapsto t}$ with the subpaths $P_{x \mapsto w} $ and $P_{w' \mapsto y}$. Path $P_{s \mapsto y}$ does not contain any vertex of path $P_{x \mapsto w}$ and path $P_{x \mapsto t}$ does not contain any vertex of path $P_{w' \mapsto y}$.
Therefore, let $s'$ be the last vertex of $P_{s \mapsto y}$ that is also in $P_{s \mapsto x}$ and $y'$ be the first vertex of $P_{s \mapsto y}$ that is also in $P_{w' \mapsto y}$. Let $x'$ be the last vertex of $P_{x \mapsto t}$ that is also in $P_{x \mapsto w}$ and $y'$ be the first vertex of $P_{x \mapsto t}$ that is also in $P_{y \mapsto t}$. The further argumentation is the same as for geodetic sets.
\end{proof}

\begin{definition}[closed geodetic sets]
A set of vertices $S$ of a series-parallel digraph $G=(V,E,s,t)$ is a {\em closed geodetic set} ({\em closed edge geodetic set}) of $G$ if every vertex $w \in V \setminus S$ (edge $e \in E$) lies on a shortest path from a vertex $u \in S$ to a vertex $v \in S$ such that $u \not= s$ or $v \not= t$.
\end{definition}

Every closed geodetic set (closed edge geodetic set) $S$ of a series-parallel digraph $G=(V,E,s,t)$ is a geodetic set (edge geodetic set, respectively) of $G$. If $(s,t)$ is an edge of $G$ then $G$ has no closed edge geodetic set. This means that every closed edge geodetic set has at least three vertices. Closed geodetic sets (closed edge geodetic set) are introduced solely for the computation of the minimum geodetic sets (minimum edge geodetic sets) in parallel composition, as stated in theorem \ref{Theorem05} 3.(b) and \ref{Theorem06} 3.(b). For this purpose, they are computed at each step.

Let $\text{mgs}(G)$ be the number of vertices of a minimum geodetic set of $G$, $\text{cmgs}(G)$ be the number of vertices of a minimum closed geodetic set of $G$ and $\text{sp}(G)$ be the length of a shortest path from source vertex $s$ to target vertex $t$ in $G$.

\begin{theorem}
\label{Theorem05}
The number of vertices in a minimum geodetic set of a series-parallel digraph can be computed as follows.

\begin{enumerate}
    \item (for the initial series-parallel digraph $I_{(s,t)}$)
        \begin{enumerate}
            \item $\text{sp}(I_{(s,t)}) = 1$,
            \item $\text{mgs}(I_{(s,t)}) = 2$ and
            \item $\text{cmgs}(I_{(s,t)}) = 2$.
        \end{enumerate}

    \item (for a series composition of two series-parallel digraph $G$ and $G'$)
    \begin{enumerate}
        \item $\text{sp}(G \scomp G') = \text{sp}(G) + \text{sp}(G')$,
        \item $\text{mgs}(G \scomp G') = \text{mgs}(G) + \text{mgs}(G') -2$ and
        \item $\text{cmgs}(G \scomp G') = \left\{
        \begin{array}{ll}
            \text{mgs}(G)+\text{mgs}(G')-2 & \text{if }
            \begin{array}[t]{ll}
            & ((\text{mgs}(G) > 2) \wedge (\text{mgs}(G') > 2)) \\
             \vee & (\text{mgs}(G) = \text{cmgs}(G) > 2) \\
             \vee & (\text{mgs}(G') = \text{cmgs}(G') > 2) \\
            \end{array} \\
            \text{mgs}(G)+\text{mgs}(G')-1 & \text{otherwise} \\
        \end{array}\right.$.
    \end{enumerate}
    
    \item (for a parallel composition of two series-parallel digraph $G$ and $G'$)
    \begin{enumerate}
        \item $\text{sp}(G \pcomp G') = \min\left\{\text{sp}(G),\text{sp}(G')\right\}$,
        \item $\text{mgs}(G \pcomp G') = \left\{\begin{array}{ll}
            \text{mgs}(G) + \text{mgs}(G') -2 & \text{if } \text{sp}(G) = \text{sp}(G') \\
            \text{mgs}(G) + \text{cmgs}(G') -2 & \text{if } \text{sp}(G) < \text{sp}(G') \\
            \text{cmgs}(G) + \text{mgs}(G') -2 & \text{if } \text{sp}(G) > \text{sp}(G') \\
        \end{array} \right.$
        \item $\text{cmgs}(G \pcomp G') = \text{cmgs}(G) + \text{cmgs}(G') -2$
    \end{enumerate}
    
\end{enumerate}
\end{theorem}

\begin{proof}
The first case is obvious. Since the initial series-parallel digraph $I_{(s,t)}$ is explicitly defined, the values can be easily verified. The calculations of the number of edges on a shortest path from $s$ to $t$ in all three sub-cases 1.(a), 2.(a) and 3.(a) are also obvious. Therefore, we only consider the cases 2.(b), 2.(c), 3.(b) and 3.(c) in more detail.

\begin{enumerate}
\setcounter{enumi}{1}
\item (series composition)

Every geodetic set $T$ of $G \scomp G'$ has at least $\text{mgs}(G) + \text{mgs}(G') -2$ vertices, because $S=(T \cap V(G)) \cup \{t\}$ is a geodetic set of $G$ and $S'=(T \cap V(G')) \cup \{s'\}$ is a geodetic set of $G'$. This implies that $|T| \geq |S|+|S'|-2 \geq \text{mgs}(G) + \text{mgs}(G') -2$.

\begin{enumerate}
    \setcounter{enumii}{1}
    \item  Let $S$ be a minimum geodetic set of $G$, $S'$ be a minimum geodetic set of $G'$ and $T^{\text{\,-} t,\text{\,-} s'} = (S \setminus \{t\}) \cup (S' \setminus \{s'\})$.

    \medskip
    Vertex set $T^{\text{\,-} t,\text{\,-} s'}$ is a geodetic set of $G \scomp G'$ with $\text{mgs}(G)+\text{mgs}(G')-2$ vertices because in $G \scomp G'$
    \begin{enumerate}
    \item every vertex $u \in V(G) \setminus T^{\text{\,-} t,\text{\,-} s'}$, $u \not = t$, is on a shortest path from a vertex of $S \setminus \{t\}$ to a vertex of $(S \setminus \{t\}) \cup \{t'\}$,
    \item every vertex $u \in V(G') \setminus T^{\text{\,-} t,\text{\,-} s'}$, $u \not = s'$, is on a shortest path from a vertex of $(S' \setminus \{s'\}) \cup \{s\}$ to a vertex of $S' \setminus \{s'\}$ and
    \item vertex $t$ is on a shortest path from vertex $s$ to vertex $t'$.
    \end{enumerate}
    
    \item  Let $S$ be a minimum geodetic set of $G$, $S'$ be a minimum geodetic set of $G'$ and $T^{\text{\,-} s'} = S \cup (S' \setminus \{s'\})$.
    
    \medskip
    Vertex set $T^{\text{\,-} s'}$ is a closed geodetic set of $G \scomp G'$ with $\text{mgs}(G)+\text{mgs}(G')-1$ because every vertex of $V(G \scomp G') \setminus T^{\text{\,-} s'}$ is on a shortest path in $G \scomp G'$ from a vertex of $S$ to a vertex of $S$ or from a vertex of $(S' \setminus \{s'\}) \cup \{t\}$ to a vertex of $(S' \setminus \{s'\}) \cup \{t\}$.

    \medskip
    Sometimes, however, vertex set $T^{\text{\,-} t,\text{\,-} s'}$ as defined in case (b) is already a closed geodetic set of $G \scomp G'$ with $\text{mgs}(G)+\text{mgs}(G')-2$ vertices.
    This is the case if $S$ and $S'$ contain more vertices than one source vertex and one target vertex, i.e., $|S| > 2$ and $|S'| > 2$.
    It is also the case if $S$ or $S'$ contain more vertices than one source vertex and one target vertex and is itself a closed geodetic set. This is the case if and only if $\text{mgs}(G) = \text{cmgs}(G) > 2$ or $\text{mgs}(G') = \text{cmgs}(G') > 2$, respectively.

\end{enumerate}
\item (parallel composition)

Every geodetic set $T$ of $G \pcomp G'$ has at least $\text{mgs}(G) + \text{mgs}(G') -2$ vertices because $S=T \cap V(G)$ is a geodetic set of $G$, $S'=(T  \cap V(G')) \cup \{s',t'\}$ is a geodetic set of $G'$ and $|T| = |S|+|S'|-2 \geq \text{mgs}(G) + \text{mgs}(G') -2$. Here $|T| = |S|+|S'|-2$ because $s'$ and $t'$ are not vertices of $G \pcomp G'$ and thus not vertices of $T$.

\begin{enumerate}
    \setcounter{enumii}{1}
    
    \item
    \begin{enumerate}
        \item
        Let $\text{sp}(G) = \text{sp}(G')$, $S$ be a minimum geodetic set of $G$, $S'$ be a minimum geodetic set of $G'$ and $T^{\text{\,-} s',\text{\,-} t'} = S \cup (S' \setminus \{t',s'\})$.

        \medskip
        Vertex set $T^{\text{\,-} s',\text{\,-} t'}$ is a geodetic set of $G \pcomp G'$ with $\text{mgs}(G)+\text{mgs}(G')-2$ vertices because every vertex $u \in V(G \pcomp G') \setminus T^{\text{\,-} s',\text{\,-} t'}$ is on a shortest path from a vertex of $S$ to a vertex of $S$ if $u \in V(G)$, or from a vertex of $(S' \setminus \{s',t'\}) \cup \{s,t\}$ to a vertex of $(S' \setminus \{s',t'\}) \cup \{s,t\}$ if $u \in V(G')$.
    
        \item Let $\text{sp}(G) < \text{sp}(G')$, $S$ be a minimum geodetic set of $G$, $S'_c$ be a minimum closed geodetic set of $G'$ and $T_c^{\text{\,-} s',\text{\,-}t'} = S \cup (S'_c \setminus \{s',t'\})$.
        Every minimum geodetic set for $G \pcomp G'$ has at least $|S|+|S'|-2 = \text{mgs}(G) + \text{cmgs}(G') -2$ vertices, since the pair $s,t$ cannot cover any vertices of $G'$.

        \medskip
        Vertex set $T_c^{\text{\,-} s',\text{\,-}t'}$ is a geodetic set of $G \pcomp G'$ with $\text{mgs}(G)+\text{cmgs}(G')-2$ vertices because every vertex $u \in V(G \pcomp G') \setminus T_c^{\text{\,-} s',\text{\,-}t'}$ is on a shortest path
        \begin{enumerate}
            \item from a vertex of $S$ to a vertex of $S$, if $u \in V(G)$ or
            \item from a vertex of $(S'_c \setminus \{s',t'\}) \cup \{s\}$ to a vertex of $S'_c \setminus \{s',t'\}$ or from a vertex of $S'_c \setminus \{s',t'\}$ to a vertex of $(S'_c \setminus \{s',t'\}) \cup \{t\}$, if $u \in V(G')$.
        \end{enumerate}
        \medskip
        Vertex set $T_c^{\text{\,-} s',\text{\,-}t'}$ is a minimum geodetic set of $G \pcomp G'$ because $R = T_c^{\text{\,-} s',\text{\,-}t'} \cap V(G)$ is a geodetic set of $G$, $R'_c = (T_c^{\text{\,-} s',\text{\,-}t'} \cap V(G')) \cup \{s',t'\}$ is a closed geodetic set of $G'$ and $|T_c^{\text{\,-} s',\text{\,-}t'}| = |R| + |R'_c| - 2$.
        \item If $\text{sp}(G) > \text{sp}(G')$ then exchange $G$ and $G'$ and consider the case where $\text{sp}(G) < \text{sp}(G')$.
         
    \end{enumerate}
    \item Let $S_c$ be a minimum closed geodetic set of $G$, $S'_c$ be a minimum closed geodetic set of $G'$ and $T_{cc}^{\text{\,-} s',t'} = S_c \cup (S'_c \setminus \{s',t'\})$.

    \medskip
    
    Vertex set $T_{cc}^{\text{\,-} s',t'}$ is a  closed geodetic set of $G \pcomp G'$ because every vertex $u \in V(G \pcomp G') \setminus T_{cc}^{\text{\,-} s',t'}$ is on a shortest path
    \begin{enumerate}
        \item from a vertex of $S_c$ to a vertex of $S_c \setminus \{t\}$ or from a vertex of $S_c \setminus \{s\}$ to a vertex of $S_c$ , if $u \in V(G)$, or
        \item from a vertex of $(S'_c \setminus \{s',t'\}) \cup \{s\}$ to a vertex of $S'_c \setminus \{s',t'\}$ or from a vertex of $S'_c \setminus \{s',t'\}$ to a vertex of $(S' \setminus \{s',t'\}) \cup \{t\}$, if $u \in V(G')$.
    \end{enumerate}
    
    Vertex set $T_{cc}^{\text{\,-} s',t'}$ is a minimum closed geodetic set of $G \pcomp G'$ because $R_c = T_{cc}^{\text{\,-} s',t'} \cap V(G)$ is a minimum closed geodetic set of $G$, $R'_c = (T_{cc}^{\text{\,-} s',t'} \cap V(G')) \cup \{s',t'\}$ is a minimum closed geodetic set of $G'$ and $|T_{cc}^{\text{\,-} s',t'}| = |R_c| + |R'_c| - 2$.

\end{enumerate}

\end{enumerate}
    
\end{proof}

\begin{figure}[hbt]
\centerline{\includegraphics[width=200pt]{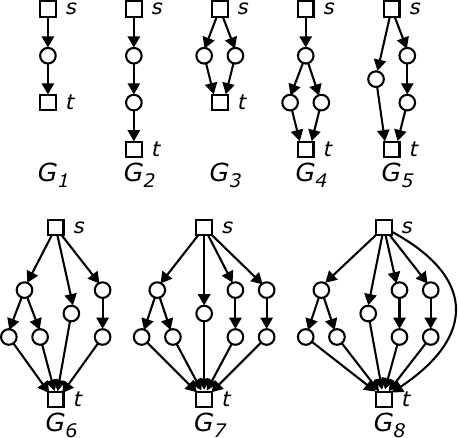}}
\caption{Eight series-parallel digraphs $G_1,\ldots,G_8$ defined as follows. \newline
$\begin{array}{llll}
G_1 = I_{(s,t)} \scomp I_{(s,t)}, & G_2 = G_1 \scomp I_{(s,t)}, &
G_3 = G_1 \pcomp G_1, & G_4 = I_{(s,t)} \scomp G_3, \\
G_5 = G_1 \pcomp G_2, & G_6 = G_4 \pcomp G_5, &  G_7 = G_6 \pcomp G_2, & G_8 = G_7 \pcomp I_{(s,t)}.
\end{array}$
}
\label{Figure12}
\end{figure}

We illustrate the calculation rules from theorem \ref{Theorem05} using the example graphs of figure \ref{Figure12}.
$$
\begin{array}{l|l|l|l}
G & \text{sp}(G) & \text{mgs}(G) &  \text{cmgs}(G) \\
\hline
I_{(s,t)}                        & 1.(a)           & 1.(b)     & 1.(c) \\
                                 & 1               & 2         & 2 \\
\hline
G_1 = I_{(s,t)} \scomp I_{(s,t)} & 2.(a)           & 2.(b)     & 2.(c), \text{ otherwise} \\
                                 & 1+1 = 2         & 2+2-2 = 2 & 2+2-1 = 3 \\
\hline
G_2 = G_1 \scomp I_{(s,t)}       & 2.(a)           & 2.(b)     & 2.(c), \text{ otherwise} \\
                                 & 2+1 = 3         & 2+2-2 = 2 & 2+2-1 = 3 \\
\hline
G_3 = G_1 \pcomp G_1             & 3.(a)           & 3.(b), \text{sp}(G_1) = \text{sp}(G_1) & 3.(c) \\
                                 & \min\{2,2\} = 2 & 2+2-2 = 2 & 3+3-2 = 4 \\
\hline
G_4 = I_{(s,t)} \scomp G_3       & 2.(a)           & 2.(b)     & 2.(c), \text{ otherwise} \\
                                 & 1+2 = 3         & 2+2-2 = 2 & 2+2-1 = 3 \\

\hline
G_5 = G_1 \pcomp G_2             & 3.(a)           & 3.(b), \text{sp}(G_1) < \text{sp}(G_2)     & 3.(c) \\
                                 & \min\{2,3\} = 2 & 2+3-2 = 3 & 3+3-2 = 4 \\
\hline
G_6 = G_4 \pcomp G_5             & 3.(a)           & 3.(b), \text{sp}(G_4) > \text{sp}(G_5)     & 3.(c) \\
                                 & \min\{3,2\} = 2 & 3+3-2 = 4 & 3+4-2 = 5 \\
\hline
G_7 = G_6 \pcomp G_2             & 3.(a)           & 3.(b), \text{sp}(G_6) > \text{sp}(G_2)     & 3.(c) \\
                                 & \min\{2,3\} = 2 & 5+2-2 = 5 & 5+3-2 = 6 \\
\hline
G_8 = G_7 \pcomp I_{(s,t)}       & 3.(a)           & 3.(b), \text{sp}(G_7) > \text{sp}(I_{(s,t)})     & 3.(c) \\
                                 & \min\{2,1\} = 1 & 6+2-2 = 6 & 6+2-2 = 6 \\
\end{array}
$$

The same methodology can be used, with minor adjustments, to compute minimum edge geodetic sets. However, special attention is required when dealing with series compositions with graphs where the shortest path has length one. 

Let $\text{megs}(G)$ be the number of vertices of an minimum edge geodetic set of $G$, $\text{cmegs}(G)$ be the number of vertices of a minimum closed edge geodetic set of $G$. If $G$ has no closed edge geodetic set then $\text{cmegs}(G) = \text{undefined}$. Since the closed edge geodetic set is only required for the parallel composition and only for the graph with the longer shortest path, the process can still be formulated, even in cases where the closed edge geodetic set is undefined.

\begin{theorem}
\label{Theorem06}
The number of vertices in a minimum edge geodetic set of a series-parallel digraph can be computed as follows.

\begin{enumerate}
    \item (for the initial series-parallel digraph $I_{(s,t)}$)
        \begin{enumerate}
            \item $\text{sp}(I_{(s,t)}) = 1$,
            \item $\text{megs}(I_{(s,t)}) = 2$ and
            \item $\text{cmegs}(I_{(s,t)}) = \text{\rm undefined}$.
        \end{enumerate}

    \item (for a series composition of two series-parallel digraph $G$ and $G'$)
    \begin{enumerate}
        \item $\text{sp}(G \scomp G') = \text{sp}(G) + \text{sp}(G')$,
        \item $\text{megs}(G \scomp G') = \text{megs}(G) + \text{megs}(G') -2$ and
        \item $\text{cmegs}(G \scomp G') = \left\{
        \begin{array}{ll}
            \text{megs}(G)+\text{megs}(G')-2 & \text{if }
            \begin{array}[t]{ll}
            & ((\text{megs}(G) > 2) \wedge (\text{megs}(G') > 2)) \\
             \vee & (\text{megs}(G) = \text{cmegs}(G)) \\
             \vee & (\text{megs}(G') = \text{cmegs}(G')) \\
            \end{array} \\
            \text{megs}(G)+\text{megs}(G')-1 & \text{otherwise} \\
        \end{array}\right.$.
    \end{enumerate}
    
    \item (for a parallel composition of two series-parallel digraph $G$ and $G'$)
    \begin{enumerate}
        \item $\text{sp}(G \pcomp G') = \min\left\{\text{sp}(G),\text{sp}(G')\right\}$,
        \item $\text{megs}(G \pcomp G') = \left\{\begin{array}{ll}
            \text{megs}(G) + \text{megs}(G') -2 & \text{if } \text{sp}(G) = \text{sp}(G') \\
            \text{megs}(G) + \text{cmegs}(G') -2 & \text{if } \text{sp}(G) < \text{sp}(G') \\
            \text{cmegs}(G) + \text{megs}(G') -2 & \text{if } \text{sp}(G) > \text{sp}(G') \\
        \end{array} \right.$
        \item $\text{cmegs}(G \pcomp G') = \left\{
            \begin{array}{ll}
            \text{\rm undefined} & \text{if }
                \begin{array}[t]{ll}
                    & (\text{cmegs}(G) = \text{\rm undefined}) \\
                    \vee & (\text{cmegs}(G') = \text{\rm undefined}) \\
                \end{array} \\
             \text{cmegs}(G) + \text{cmegs}(G') -2 & \text{otherwise}
            \end{array}\right.$
    \end{enumerate}
    
\end{enumerate}
\end{theorem}

\begin{proof}
It is only necessary to examine cases 2.(b), 2.(c), 3.(b) and 3.(c) in more detail. Case 1 is obvious, and cases 2.(a) and 3.(a) correspond to those in theorem \ref{Theorem05}.
\begin{enumerate}
\setcounter{enumi}{1}
\item (series composition)

Every edge geodetic set $T$ of $G \scomp G'$ has at least $\text{megs}(G) + \text{megs}(G') -2$ vertices, because $S=(T \cap V(G)) \cup \{t\}$ is an edge geodetic set of $G$ and $S'=(T \cap V(G')) \cup \{s'\}$ is an edge geodetic set of $G'$. This implies that $|T| \geq |S|+|S'|-2 \geq \text{megs}(G) + \text{megs}(G') -2$.

\begin{enumerate}
    \setcounter{enumii}{1}
    
    \item  Let $S$ be a minimum edge geodetic set of $G$, $S'$ be a minimum edge geodetic set of $G'$ and $T^{\text{\,-} t,\text{\,-} s'} = (S \setminus \{t\}) \cup (S' \setminus \{s'\})$.

    \medskip
    Vertex set $T^{\text{\,-} t,\text{\,-} s'}$ is an edge geodetic set of $G \scomp G'$ with $\text{megs}(G)+\text{megs}(G')-2$ vertices because in $G \scomp G'$
    \begin{enumerate}
    \item every edge $e \in E(G)$ is on a shortest path from a vertex of $S \setminus \{t\}$ to a vertex of $(S \setminus \{t\}) \cup \{t'\}$ and
    \item every edge $e \in E(G')$ is on a shortest path from a vertex of $(S' \setminus \{s'\}) \cup \{s\}$ to a vertex of $S' \setminus \{s'\}$.
    \end{enumerate}
    
    \item  Let $S$ be a minimum edge geodetic set of $G$, $S'$ be a minimum edge geodetic set of $G'$ and $T^{\text{\,-} s'} = S \cup (S' \setminus \{s'\})$.
    
    \medskip
    Vertex set $T^{\text{\,-} s'}$ is a closed edge geodetic set of $G \scomp G'$ with $\text{megs}(G)+\text{megs}(G')-1$ because every edge of $E(G \scomp G') $ is on a shortest path in $G \scomp G'$ from a vertex of $S$ to a vertex of $S$ or from a vertex of $(S' \setminus \{s'\}) \cup \{t\}$ to a vertex of $(S' \setminus \{s'\}) \cup \{t\}$.

    \medskip
    Sometimes, however, vertex set $T^{\text{\,-} t,\text{\,-} s'}$ as defined in case (b) is already a closed edge geodetic set of $G \scomp G'$ with $\text{megs}(G)+\text{megs}(G')-2$ vertices.
    This is the case if $S$ and $S'$ contain more vertices than one source vertex and one target vertex, i.e., $|S| > 2$ and $|S'| > 2$.
    It is also the case if $S$ or $S'$ contain more vertices than one source vertex and one target vertex and is itself a closed edge geodetic set. This is the case if and only if $\text{megs}(G) = \text{cmegs}(G) > 2$ or $\text{megs}(G') = \text{cmegs}(G') > 2$, respectively. Recall that, by definition, graphs with sp(G)= 1 do not have closed edge geodetic sets.

\end{enumerate}

\item (parallel composition)

Every edge geodetic set $T$ of $G \pcomp G'$ has at least $\text{megs}(G) + \text{megs}(G') -2$ vertices because $S=T \cap V(G)$ is an edge geodetic set of $G$, $S'=(T  \cap V(G'))\cup \{s',t'\}$ is an edge geodetic set of $G'$ and $|T|=|S|+|S'|-2 \geq \text{megs}(G) + \text{megs}(G') -2$.

\begin{enumerate}
    \setcounter{enumii}{1}
    
    \item
    \begin{enumerate}
        \item
        Let $\text{sp}(G) = \text{sp}(G')$, $S$ be a minimum edge geodetic set of $G$, $S'$ be a minimum edge geodetic set of $G'$ and $T^{\text{\,-} s',\text{\,-} t'} = S \cup (S' \setminus \{t',s'\})$.

        \medskip
        Vertex set $T^{\text{\,-} s',\text{\,-} t'}$ is an edge geodetic set of $G \pcomp G'$ with $\text{megs}(G)+\text{megs}(G')-2$ vertices because in $G \scomp G'$
        \begin{enumerate}
        \item every edge $e \in E(G)$ is on a shortest path from a vertex of $S$ to a vertex of $S$ and
        \item every edge $e \in E(G')$ is on a shortest Path from a vertex of $(S' \setminus \{s',t'\}) \cup \{s,t\}$ to a vertex of $(S' \setminus \{s',t'\}) \cup \{s,t\}$.
        \end{enumerate}
    
        \item Let $\text{sp}(G) <\text{sp}(G')$, $S$ be a minimum edge geodetic set of $G$, $S'_c$ be a closed minimum edge geodetic set of $G'$ and $T_c^{\text{\,-} s',{\text{\,-} t'}} = S \cup (S'_c \setminus \{t',s'\})$. Now every minimum edge geodetic set for $G \pcomp G'$ has at least $|S|+|S'|-2 = \text{mgs}(G) + \text{cmgs}(G') -2$ vertices, since the pair $s,t$ cannot cover any vertices of $G'$. Since the shortest path of G' is always greater than 1, the $\text{megs}$ is always defined.

        \medskip
        Vertex set $T_c^{\text{\,-} s',{\text{\,-} t'}}$ is an edge geodetic set of $G \pcomp G'$ with $\text{megs}(G)+\text{cmegs}(G')-2$ vertices because every edge $e \in E(G \pcomp G')$ is on a shortest path
        \begin{enumerate}
            \item from a vertex of $S$ to a vertex of $S$, if $e \in E(G)$ or
            \item from a vertex of $(S'_c \setminus \{s',t'\}) \cup \{s\}$ to a vertex of $S'_c \setminus \{s',t'\}$ or from a vertex of $S'_c \setminus \{s',t'\}$ to a vertex of $(S'_c \setminus \{s',t'\}) \cup \{t\}$, if $e \in E(G')$.
        \end{enumerate}
        \medskip
        Vertex set $T_c^{\text{\,-} s',{\text{\,-} t'}}$ is a minimum edge geodetic set of $G \pcomp G'$ because $R = T_c^{\text{\,-} s',{\text{\,-} t'}} \cap V(G)$ is an edge geodetic set of $G$, $R'_c = (T_c^{\text{\,-} s',{\text{\,-} t'}} \cap V(G')) \cup \{s',t'\}$ is a closed edge geodetic set of $G'$ and $|T_c^{\text{\,-} s',{\text{\,-} t'}}| = |R| + |R'_c| - 2$.

        \item If $\text{sp}(G) > \text{sp}(G')$ then exchange $G$ and $G'$ and consider the case where $\text{sp}(G) < \text{sp}(G')$.
        
    \end{enumerate}
    
    \item Let $S_c$ be a minimum closed edge geodetic set of $G$, $S'_c$ be a minimum closed edge geodetic set of $G'$ and $T_{cc}^{\text{\,-} s',t'} = S_c \cup (S'_c \setminus \{s',t'\})$.

    \medskip
    Vertex set $T_{cc}^{\text{\,-} s',t'}$ is a closed edge geodetic set of $G \pcomp G'$ because every edge $e \in E(G \pcomp G')$ is on a shortest path
    \begin{enumerate}
        \item from a vertex of $S_c$ to a vertex of $S_c \setminus \{t\}$ or from a vertex of $S_c \setminus \{s\}$ to a vertex of $S_c$ , if $e \in E(G)$, or
        \item from a vertex of $(S'_c \setminus \{s',t'\}) \cup \{s\}$ to a vertex of $S'_c \setminus \{s',t'\}$ or from a vertex of $S'_c \setminus \{s',t'\}$ to a vertex of $(S' \setminus \{s',t'\}) \cup \{t\}$, if $e \in E(G')$.
    \end{enumerate}
    Vertex set $T_{cc}^{\text{\,-} s',t'}$ is a minimum closed edge geodetic set of $G \pcomp G'$ because $R_c = T_{cc}^{\text{\,-} s',t'} \cap V(G)$ is a closed edge geodetic set of $G$, $R'_c = (T_{cc}^{\text{\,-} s',t'} \cap V(G')) \cup \{s',t'\}$ is a closed edge geodetic set of $G'$ and $|T_{cc}^{\text{\,-} s',t'}| = |R_c| + |R'_c| - 2$.

\end{enumerate}

\end{enumerate}
    
\end{proof}

\begin{figure}[hbt]
\centerline{\includegraphics[width=160pt]{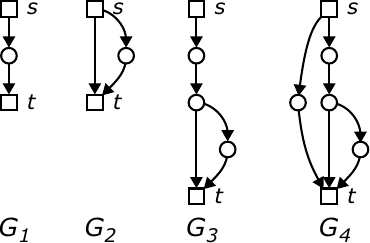}}
\caption{Four series-parallel digraphs $G_1,\ldots,G_4$ defined as follows. 
\newline
$\begin{array}{llll}
G_1 = I_{(s,t)} \scomp I_{(s,t)}, & G_2 = I_{(s,t)} \pcomp G_1, &
G_3 = G_1 \scomp G_2, & G_4 = G_1 \pcomp G_3 \\
\end{array}$
}
\label{Figure13}
\end{figure}

We illustrate the calculation rules from theorem \ref{Theorem06} using the example graphs of figure \ref{Figure13}.
$$
\begin{array}{l|l|l|l}
G & \text{sp}(G) & \text{megs}(G) &  \text{cmegs}(G) \\
\hline
I_{(s,t)}                        & 1.(a)           & 1.(b)     & 1.(c) \\
                                 & 1               & 2         & \text{undefined} \\
\hline
G_1 = I_{(s,t)} \scomp I_{(s,t)} & 2.(a)           & 2.(b)     & 2.(c), \text{ otherwise} \\
                                 & 1+1 = 2         & 2+2-2 = 2 & 2+2-1 = 3 \\
\hline
G_2 = I_{(s,t)} \pcomp G_1       & 3.(a)           & 3.(b), \text{sp}(I_{(s,t)}) < \text{sp}(G_1) & 3.(c), \text{cmegs}(I_{(s,t)}) = \text{\rm undefined} \\
                                 & \min\{1,2\} = 1 & 2+3-2 = 3 & \text{undefined} \\
\hline
G_3 = G_1 \scomp G_2             & 2.(a)           & 2.(b) & 2.(c), \text{ otherwise} \\
                                 & 2 + 1 = 3 & 2+3-2 = 3 & 2+3-1 = 4 \\
\hline
G_4 = G_1 \pcomp G_3             & 3.(a)           & 3.(b), \text{sp}(G_1) < \text{sp}(G_3)     & 3.(c), \text{ otherwise} \\
                                 & \min\{2,3\} = 2 & 2+4-2 = 4 & 3+4-2 = 5 \\

\end{array}
$$

This approach can easily be extended to compute an explicit minimum (edge) geodetic set, instead of simply determining its size by tracking the vertex sets during the computation. A detailed discussion of these aspects is omitted here and left to the interested reader.

The ideas used here do not appear to be transferable to undirected graphs. When performing the parallel composition of two undirected series-parallel graphs $G_1$ and $G_2$, shortest paths between the vertices of $G_1$ and vertices of $G_2$ must also be considered. This is not the case in the directed version, which simplifies the analysis.

\section{Conclusion}

We proved the NP-completeness of {\sc MGS}, {\sc MSGS}, and {\sc MMoGS} for directed acyclic geodesic planar graphs.
Furthermore, we proved the NP-completeness of {\sc MEGS}, {\sc MSEGS}, and {\sc MMoEGS} for directed acyclic geodesic graphs. This complements the known NP-completeness results for computing geodesic sets.

In addition, we proved that minimum geodetic sets and minimum edge geodetic sets can be computed in linear time for directed series-parallel graphs.

\bibliography{Bibliography/bibliography_references}

\end{document}